\documentclass[12pt]{article}
\usepackage{amsmath,
amsfonts,amsthm,amssymb,amsbsy,upref,color,graphicx,amscd,hyperref,makeidx,
enumerate}
\usepackage{tikz}
\usepackage[active]{srcltx}
\usepackage[latin1]{inputenc}
\usepackage{enumerate}
\usepackage[english]{babel}
\usepackage{newlfont}
\usepackage{amsbsy}
\usepackage[english]{babel}
\usepackage[center]{caption2}
\usepackage{amssymb}
\usepackage{amsmath}
\usepackage{latexsym}
\usepackage{amsthm}
\usepackage{amsthm}
\theoremstyle{plain}
\newtheorem{thm}{Theorem}

\newtheorem{lem}[thm]{Lemma}
\newtheorem{prop}[thm]{Proposition}
\newtheorem{nota}[thm]{Notation}
\newtheorem{rem}[thm]{Remark}
\newtheorem{defin}[thm]{Definition}

\begin{document}

\title{Treelike quintet systems}
\author{Simone Calamai \  \ \  \ \ Elena Rubei}

\date{\today}
\maketitle

\def\thefootnote{}
\footnotetext{ \hspace*{-0.5cm}

 }

\begin{abstract} Let $X$ be a finite set.
We give criterion to say if
 a system of trees ${\cal P}=\{T_i\}_i$ with  leaf sets  $L(T_i) \in {X
\choose 5}$ 
can be  amalgamated into a supertree, that is, if there exists a tree $T$ with
$L(T)=X$  such that $T$ restricted to $ L(T_i)$ is equal to $T_i$.  

\end{abstract}

\section{Introduction}

Phylogenetic trees  are used to represent evolutionary relationships
among some taxa in many fields, such as biology and philology. Unfortunately, 
methods to reconstruct phylogenetic trees generally do not work for large
numbers of taxa; so it would be useful  to have a criterion to say if, given a
collection of 
phylogenetic trees with overlapping sets of taxa, there exists a (super)tree
``including'' all the trees in the given collection. In this case we say that
the collection is ``compatible''.
The literature on the ``supertree problem'' is wide. We quote only few of the
known results.

One of the first results is due to Colonius and Schultze: in  \cite{C-S} they
gave a criterion to say if, given a finite set $X$, a system of trees ${\cal P}=
\{T_i\}_i$ with leaf sets $L(T_i) \in {X \choose 4}$ 
can be  amalgamated into a supertree, that is, if there exists a tree $T$ with
$L(T)=X$  such that $T$ restricted to ${L(T_i)}$ is equal to $T_i$.  Obviously,
a tree with leaf set 
$\{a,b,c,d\}$ is determined by the partition (called ``quartet'')  of
$\{a,b,c,d\}$ into the cherries; Colonius and Schultze defined three properties,
thinness, transitivity and saturation, that are necessary and sufficient for a
quartet system to be treelike.

In \cite{BBDS} the authors suggest a polynomial time algorithm that, 
given a collection of trees,  produces a supertree, if it exists and 
under some conditions. 

We quote also the papers \cite{BS} and \cite{GSS}, where the authors studied 
closure rules among compatible trees, i.e. rules that, given a compatible 
collection of trees, determine some other trees not in the original collection.

Finally, in 2012, Gr\"{u}newald 
gave a sufficient criterion for a set of binary phylogenetic trees to be
compatible; precisely he proved that, if ${\cal P}$ is a finite collection of
phylogenetic trees and the cardinality of the union of the leaf sets of the 
elements of ${\cal P}$ minus $3$ is equal to the sum of the  cardinalities 
of the set of the interior edges of the  elements of ${\cal P}$, then ${\cal
P}$ 
is compatible (see \cite{Gru}).

A possible variant of the supertree problem is to fix the cardinality of the
leaf sets of the trees in the given collection. 
In this paper we consider this problem in the case the cardinality of the leaf
set
of every tree in the given collection is $5$. 
Obviously a tree with the cardinality of the leaf set equal to $5$ is given by a
 partition (called ``quintet'')  of the leaf set  into the cherries and  the
complementary of the union of the cherries. 
We define three properties, analogous to the ones for quartets, that are
necessary and sufficient for a quintet system to be treelike.


\section{Notation and recalls}

\begin{defin}
$\bullet $ Let $Y$ be a set. A partition of $Y$ into $k$ subsets
of cardinality $n_1, \ldots ,n_k$, with $n_1 \geq \ldots \geq n_k$, is said a partition of kind $(n_1, \ldots, n_k)$.

$\bullet$ 
Let $X$ be set.

A partition of a $4$-subset $Y$ of $X$ is said a {\bf quartet}  (on $Y$) in $X$
if its kind is one of the following: $(2,2)$,  $(4)$.

A partition of a $5$-subset $Y$ of $X$ is said a {\bf quintet}  (on $Y$) in $X$
if its kind is one of the following: $(2,2,1)$, $(3,2)$, $(5)$.
 
\end{defin}

\begin{nota}
Throughout the paper, $X$ will denote a finite nonempty set.
\end{nota}

\begin{defin} 
  
 Let $T$ be a tree.

 $\bullet$ We denote by $L(T)$ the leaf set of $T$.
  
 $\bullet$ For any  $S \subset L(T)$, we denote by $T|_S$ the minimal subtree 
of $T$ whose vertex set   contains $S$.

$\bullet$ 
We say that two leaves  $i$ and $j$ of $T$ are  neighbours
if in the path from $i$ to $j$ there is only one vertex of  degree  greater than
$2$; 
furthermore, we say that  $C \subset L(T)$ is a {\bf cherry} if any $i,j \in C$
are neighbours.

$\bullet $ We say that a cherry  is {\bf complete}
 if it is not strictly  contained in another cherry.
\end{defin}

\begin{defin}
A {\bf phylogenetic $X$-tree} $(T ,\varphi ) $ is a finite tree $T$ without
vertices of degree $2$ and endowed with a bijective function 
$\varphi: X \rightarrow L(T) $.

{\bf The quartet system $S$ in $X$  associated to a phylogenetic $X$-tree} is
the 
quartet system defined as follows: for any $a,b,c,d \in X$,

$(a,b \, |\, c,d) \in S$ if and only if $\{a,b\}$ and $\{c,d\}$ are complete
cherries of $T|_{\{a,b,c,d\}}$,

$(a,b , c,d) \in S$ if and only if $T|_{\{a,b,c,d\}}$  is a star tree.

{\bf The quintet system $S$ in $X$  associated to a phylogenetic $X$-tree} is
the 
quintet system defined as follows: for any $a,b,c,d,e \in X$,

$(a,b \, |\, c,d\,|\, e) \in S$ if and only if
 $\{a,b\}$ and $\{c,d\}$ are complete cherries of $T|_{\{a,b,c,d,e\}}$,

$(a,b \, |\, c,d, e) \in S$ if and only if
 $\{a,b\}$ and $\{c,d,e\}$ are complete cherries of $T|_{\{a,b,c,d,e\}}$,

$(a,b , c,d,e) \in S$ if and only if $T|_{\{a,b,c,d,e\}}$  is a star tree.

Given a quartet system (respectively  a quintet system) $S$ in $X$ and 
 a quartet (respectively a quintet) in  $X$, we  often write either simply ``$P$'' or ``$P$ holds''
instead of writing ``$ P \in S$'' when it   is clear from the context the system
which we are referring to.

\end{defin}

\begin{defin} Let $S$ be a  quartet system  in $X$.

$\bullet$ 
We say that $S$ is {\bf saturated} if the following implication holds
for any $a_1, a_2, b_1, b_2, x \in X$:

$(a_1, a_2 \, | \, b_1, b_2) \Rightarrow (a_1, x \,|\, b_1, b_2 ) \vee (a_1, a_2
\,|\, b_1, x )$.

$\bullet$ 
We say that $S$ is {\bf transitive} if the following implication holds
for any $a_1, a_2, b_1, b_2, x \in X$:

$(a_1, x \, | \, b_1, b_2)  \wedge (a_2, x \, | \, b_1, b_2) \Rightarrow (a_1,
a_2 \,|\, b_1, b_2 ) $.

$\bullet $ 
We say that $S$ is {\bf thin} if, 
for any $4$-subset $Y$ of $X$, there exists only one quartet on $Y$ in $S$.
\end{defin}

As we have already said in the introduction, 
Colonius and Schultze characterized treelike quartet systems. The statement
we recall here is the one in  \cite{Dresslibro}.

\begin{thm} \label{CS} 
 Let $S$ be a quartet system in $X$; we have that  $S$ is the quartet system of
a phylogenetic  $X$-tree
if and only if $S$ is thin, transitive and saturated.
\end{thm}

\begin{nota} Let $a_1, a_2 ,  b_1, b_2 , b_3 \in X$ and
let $Q$ be a  quintet system in $X$.

We write $(a_1, a_2 \,|\, \overline{ b_1, b_2, b_3})$ instead of 
  $$ (a_1, a_2 \,|\,  b_1, b_2, b_3) \vee (a_1, a_2 \,|\,  b_1, b_2 \,|\, b_3) 
   \vee (a_1, a_2 \,|\,  b_1, b_3 \,|\, b_2) 
   \vee (a_1, a_2 \,|\,  b_2, b_3 \,|\, b_1) 
  $$
  
\end{nota}

\begin{defin} Let $Q$ be a  quintet system in $X$.

$\bullet$ 
We say that $Q$ is {\bf saturated} if the following implications hold
for any $a_i, b_j, c, x \in X$:

\smallskip

(i) $(a_1, a_2 \,|\, b_1, b_2 \,|\,c)  \Rightarrow $  
$(a_1, a_2 \,|\, b_1, b_2 \,|\,x)  \vee (a_1, x \,|\, b_1, b_2 \,|\,c)  \vee
(a_1, a_2 \,|\, b_1, x \,|\,c) $,

\smallskip

(ii) $(a_1, a_2 \,|\, b_1, b_2 , b_3 )  \Rightarrow (a_1, x \,|\, b_1, b_2 ,
b_3)  \vee (a_1, a_2 \,|\, \overline{ b_1, b_2, x})$

\smallskip

(iii)  $(a_1, a_2 , a_3, a_4 , a_5 )  \Rightarrow $  

$(a_1, a_2 , a_3, a_4 , x) 
 \vee (a_1, x \,|\, a_2, a_3 , a_4 )  
\vee (a_2, x \,|\, a_1, a_3 , a_4 )  
\vee (a_3, x \,|\, a_1, a_2 , a_4 )  
\vee (a_4, x \,|\, a_1, a_2 , a_3 )  $

$\bullet$
We say that $Q$ is {\bf transitive} if the following implications hold
for any $a_i, b_j, c_k, x \in X$:

(i)  $(a_1, a_2 \, |\, b_1, x \,| \, c_1) \wedge  (a_1, a_2 \, |\, b_1, x \,| \,
c_2)  \Rightarrow  (a_1, a_2 \, |\, \overline{ c_1, c_2 , b_1})$, 

(ii) $(a_1, a_2 \, |\, b_1, x \,| \, c_1) \wedge  (a_1, a_2 \, |\, b_2, x \,| \,
c_1)  \Rightarrow $ $(a_1, a_2 \, |\, b_1, b_2 \,| \, c_1) $

(iii) $(a_1, x \, |\, b_1, b_2 , b_3) \wedge  (a_2, x \, |\, b_1, b_2 , b_3) 
\Rightarrow $ $(a_1, a_2 \, |\, b_1, b_2 , b_3) $

(iv) $(a_1, a_2 \, |\, b_1, b_3 , x) \wedge  (a_1, a_2 \, |\, b_2, b_3, x) 
\Rightarrow $ $(a_1, a_2 \, |\, b_1, b_2 , b_3) \vee  (a_1, a_2 \, |\, b_1 , b_2
\,|\, b_3)  $

(v) $(a_1, a_2 \, |\, b_1,x,  b_2 ) \wedge  (a_1, a_2 \, |\, b_1, x \,|\, b_3) 
\Rightarrow $ $(a_1, a_2 \, |\, b_1, b_2 \,|\, b_3)$

(vi) $(a_1, a_2 \, |\, b_1, b_2 \,|\, x)  \wedge 
(a_1, a_2 \, |\, b_1,  b_3,x )    \Rightarrow $ $(a_1, a_2 \, |\, b_1, b_2 \,|\,
b_3)$

$\bullet $ 
We say that $Q$ is {\bf thin} if, 
for any $5$-subset $Y$ of $X$, there exists only one quintet on $Y$ in $Q$ and,
for any $a,b,c,d,x,y \in X$,

(i) $(a,b \,| \,c,x \,|\, d) \wedge (a,c\,| \,b,y \,| \, d)$ is impossible, 

(ii)  $(a,b \,| \,c, d,x) \wedge (a,y \,| \,b, c,d)$ is impossible, 

(iii) $(a,b \,| \,c,x \,|\, d) \wedge (a,c,d \,| \,b,y)$ is impossible,

(iv)  $(a,x \,| \,b,c, d) \wedge (a,d \,| \,b,c \,|\, y)$ is impossible.

\end{defin}

Both for quartet systems and quintet systems, we will write TTS instead of thin, transitive and saturated.

\section{Characterization of treelike quintet systems}

Our aim is the prove that a quintet system is treelike if and only if 
it is TTS.

First of all, we need to define the quartet system associated to a quintet system
and the quintet system associated to a quartet system.

\begin{defin}\label{defi:quartet S coming from quintet Q}
 Given a TTS quintet system $Q$ on  $X$,
 let $S$ be the quartet  system  defined as follows: for any $a,b,c,d\in X$,
we have that 
 $(a,b\, |\,c,d) \in S$ if and only if there exists $y\in X$ for which at least one of
the following instances
 occurs:
 
 $ (i) \; (a,b\, |\,c,d\, |\,y) \in Q$,
 
 $ (ii) \; (a,b\, |\,c,d,y) \in Q$,
 
 $(iii) \; (a,b\, |\,c,y\, |\,d) \in Q $,
 
 $(iv) \; (a,b\, |\,d,y\, |\,c) \in Q$,
 
 $(v) \; (c,d\, |\,b,y\, |\,a) \in Q$,
 
 $(vi) \; (c,d\, |\,a,y\, |\,b) \in Q$,
 
 $(vii) \; (a,b,y\, |\,c,d) \in Q$.
 
We say that $S$ is {\bf the quartet system associated to the quintet system $Q$}.
\end{defin}

\begin{defin}
Let $S$ be a TTS quartet system in $X$. 
Let $Q'$ be the quintet system in $X$ defined as follows:

$(a, b\, |\, c,d \,| \, e ) \in Q' $ if and only if $(a, b\, |\, c,d  ) , 
(a, b\, |\, c,e  ), 
(a, e\, |\, c,d  ) \in S $

$(a, b\, |\, c,d , e ) \in Q' $ if and only if $(a, b\, |\, c,d  ) , 
(a, b\, |\, c,e  ), 
(a, e, c,d  ), (b, c,d  ,e) \ \in S $

$(a_1, a_2, a_3,a_4 , a_5 ) \in Q' $ if and only if $(a_1,...., \hat{a}_i,.... ,
a_5 )  \in S $ for any $ i \in \{1,...., 5\}$.

We say that $Q'$ is {\bf the quintet system associated to the quartet system}
$S$. 
\end{defin}

The sketch of the proof of our result is the following: given a TTS quintet
system $Q$, we will show that the associated quartet system $S$ is TTS;
so there exists a phylogenetic $X$-tree $(T,\varphi)$ inducing $S$. 
We will show that the quintet system associated to $(T, \varphi)$ 
is exactly $Q$ and this will end the proof.

First, we have 
to state two lemmas which will be useful in the remainder of the paper.

\begin{lem}\label{lem: unione di lemmi}
 Let $Q$ be a TTS  quintet system  in $X$. For any  $a,b,c,d,x,y \in X$, it is
not possible the simultaneous  occurrence of $(A) \; (a,b \, | \, c,x \, |\, d
)$ and any of the following:

$(B) \; (a, c \,|\, b, d \, |\, y)$,

$(C) \; (a , c \, | \, d , y \, |\, b)$,

$(D) \; (b , d \, | \, c , y \, |\, a)$,

$(E) \; (y , c \, |\, a ,b , d)$,

$(F) \; (a , c \, |\, b ,d, y)$,

$(G) \; (a , d \, |\, b ,c , y)$,

$(H) \; (d , y \, |\, a ,b, c)$.
\end{lem}
\begin{proof}
 As $Q$ is saturated, $(A)$ implies at least one of the following cases:
 
$(A.1)  \; (a,y \, |\, c,x \,| \, d)  \;\wedge \;(y,b \, |\, c,x \,| \, d),  $  
  
$(A.2)  \; (a,b \, |\, c,y \,| \, d) \;\wedge \;(a,b \, |\, x,y \,| \, d),$ 
  
$(A.3)  \; (a,b \, |\, c,x \,| \, y)$.

We claim that $(A)\wedge(B)$ cannot hold. As $Q$ is saturated,
$(B)$ implies at least one of the following:

$(B.1)  \; (a,x \, |\, b,d \,| \, y)  \;\wedge \;(x,c \, |\, b,d \,| \, y),  $

$(B.2)  \; (a,c \, |\, x , d \,| \, y)\;\wedge \;(a,c \, |\, b,x \,| \, y),$ 
  
$(B.3)  \; (a,c \, |\, b,d \,| \, x)$.

Since $Q$ is thin, each of $(A)\wedge(B.3)$,
$(A.1)\wedge (B.1)$, $(A.1)\wedge (B.2)$,  $(A.2)\wedge (B)$, $(A.3)\wedge (B.2)$ is impossible. As $Q$ is
transitive, $(A)\wedge(A.3)$ implies
$(a, b \, | \, \overline{d, x , y})$ which contradicts $(B.1)$; thus the claim is proved.

We now show that $(A)\wedge (C)$ cannot hold. As $Q$ is saturated, $(C)$ implies at least one of the following:

$(C.1) \; (a, x \, | \, d , y \, | \, b) \; \wedge \;(x, c \, | \, d , y \, |
\,
b) ,$

$(C.2) \; (a, c \, | \, d , x \, | \, b)$,

$(C.3) \; (a, c \, | \, d , y \, | \, x)$.

Since $Q$ is thin, each of  $(A)\wedge (C.2)$,
$(A.1)\wedge (C.1)$, $(A.1)\wedge (C.3)$,  $(A.2)\wedge(C)$ is impossible.
As $Q$ is transitive, $(A)\wedge (A.3)$ implies $ (a,b \, |\, \overline{c, d,y}
)$, which contradicts $(C)$, yielding the claim. 

We prove that $(A)\wedge (D)$ cannot hold. As $Q$ is saturated, $(D)$ implies at least one of the following:

$(D.1) \; (x, d \, | \, c , y \, | \, a) \; \wedge \; (b, x \, | \, c , y \, |
\, a)$,

$(D.2) \; (b, d \, | \, c , x \, | \, a)$,

$(D.3) \; (b, d \, | \, c , y \, | \, x) $.

Since $Q$ is thin, it is impossible to have each of $(A)\wedge(D.2)$,
$(A.1)\wedge(D.1)$, $(A.1)\wedge (D.3)$, $(A.2)\wedge (D)$. As $Q$ is
transitive, $(A)\wedge(A.3)$ implies $(a,b \,
|\,
\overline{c, d,y} )$, which contradicts $(D)$, yielding the claim.

We claim that $(A)\wedge(E) $ is impossible.  As $Q$ is saturated, $(E)$ implies

$   (x, c \, | \, a , b , d)\; (E.1) \quad \vee \quad [(y , c\, | \, \overline{
a , b , x}) \: (E.2) \; \wedge \; (y , c \, | \, \overline{ a ,d , x}) \:
(E'.2)] $.

The thinness of $Q$ excludes each of  $(A)\wedge
(E.1)$, $(A.1)\wedge (E'.2)$, $(A.2)\wedge (E)$,  $(A.3)\wedge
(E.2)$, yielding the claim.

We claim that $(A)\wedge (F)$ cannot hold. As $Q$ is saturated, $(F)$ implies at least  one of the following:

$(F.1)\; (a,x \, |\,b,d, y) \; \wedge \; (c,x \, |\, b,d,y)  $, 

$(F.2)\; (a,c \, |\, \overline{b,d,x}) $.

As $Q$ is thin, each of $(A)\wedge(F.2)$, $(A.1)\wedge(F.1)$, 
$(A.2)\wedge(F.1)$ is impossible. Moreover, since $Q$ is transitive,  $(A.3)\wedge(A)$
implies $(a,b \, |\,
\overline{d,x,y}) $, which contradicts $(F.1)$ by the thinness of $Q$, so the
claim follows.

We claim that $(A)\wedge (G)$ cannot hold. As $Q$ is saturated, $(G)$ implies at least one of the following:

$(G.1)\; (a, x \, | \, b ,c ,y)\; \wedge\;  (d, x \, | \, b ,c ,y)$,

$(G.2)\; (a , d \, | \, \overline{b, c, x})$.

By the thinness of $Q$, each of $(A)\wedge(G.2)$, $(A.1)\wedge
(G.1)$, $(A.2)\wedge (G)$,  $(A.3)\wedge (G.1)$ is impossible, so we get the claim.

We claim that $(A)\wedge(H)$ cannot hold. As $Q$ is saturated, $(H)$ implies at least  one of the following:

$(H.1)\; (d, x \, | \, a, b ,c)$,

$(H.2)\; (d, y \, | \, \overline{a, c, x})$.

By the thinness of $Q$, each of $(A)\wedge(H.1)$,
$(A.1)\wedge(H.2)$,  $(A.2)\wedge (H)$ is impossible. Moreover, since $Q$ is transitive, 
$(A.3)\wedge(A)$
implies $(a,b \, |\,
\overline{d,c,y}) $, which contradicts $(H)$ by the thinness of $Q$, so the
claim follows.
\end{proof}

\begin{lem}\label{lem: osservazione nuova}
 Let $Q$ be a TTS  quintet system  in $X$. For any  $a,b,c,d,x,y \in X$, it is
not possible the simultaneous occurence of 
 $  (A) \;\;\; (a, c \, | \, b , d , y )   $ and any of the following:

$(B) \;(b , x \, | \, a ,c ,d)$, 

$(C) \; (a,b \, |\, c,d \,| \, x)  $.
\end{lem}

\begin{proof}
 As $Q$ is saturated, 
 $(A)$ implies at least one of the following:

$(A.1) \; (a , x \, |\, b, d, y) \wedge (c, x \, |\, b, d, y)$,

$(A.2) \; (a ,c \, | \, \overline{b , d, x})$.

and $(B)$ implies  at least one of the following:
 
 $(B.1)  \;  (b, y  \,|\, a,c ,  d), $ 

$ (B.2) \; (b, x  \,|\, \overline{ a, d, y})   $.
 
As $Q$ is thin, each of $(B)\wedge (A.2)$, $(B.1)\wedge (A)$,
$(B.2)\wedge (A.1)$ is impossible, which concludes the proof that 
$(A) \wedge (B)$ cannot hold.

Since $Q$ is saturated, from $(C)$ we get at least one of the following:
 
$(C.1) \; (a , y \, | \, c , d \, |\, x)\; \wedge \;(b , y \, | \, c , d \,
|\,x)$,

$(C.2) \; (a, b \, | \, c , y \, | \, x ) \; \wedge \; (a, b \, | \, d , y \, |
\, x )$,

$(C.3) \; (a , b \, | \, c , d \, | \, y)$.

By the thinness of $Q$, each of  $(C)\wedge(A.2)$,
$(C.1)\wedge(A.1)$, $(C.2)\wedge(A.1)$ cannot hold. 
Moreover $(C.3)\wedge (C)$ implies $(a,b \;|\; \overline{d, x,y})$, which contradicts $(A.1)$ and this conludes the proof that 
$(A) \wedge (C)$ cannot hold.
\end{proof}

\begin{prop} \label{esistequalsiasi}
Let $Q$ be a TTS quintet system  in $X$.  Let $a_1, a_2, b_1,b_2 \in X$.

There exists $x \in X$ such that at least one of the following holds:

$(1) \; (a_1, a_2 \, |\, b_1, b_2 \, | \, x)  $, 

$(2) \;  (a_1, a_2 \, |\, b_1, b_2 , x)  $, 

$(3) \; (a_1, a_2 , x \, |\, b_1, b_2 )  $, 

$(4) \;(a_1, a_2 \, |\, b_1, x \, |\, b_2)  $, 

$(5) \; (a_1, a_2 \, |\, b_2, x \, |\, b_1)  $, 

$(6) \;  (a_1, x \, |\, b_1, b_2 \, |\, a_2)  $, 

$(7) \;  (a_2, x \, |\, b_1, b_2 \, |\, a_1)  $

if and only if for any $x \in X$ at least one of (1),....., (7) holds.

\end{prop}

\begin{proof}
$\Leftarrow $ Obvious.

$\Rightarrow$ Suppose, contrary to our claim, that 
there exists $y \in X$ such that no one of (1),...., (7) holds (with $y$ instead
of $x$).
So we must have at least one of the following:

$(8) \;(a_1, a_2 ,  b_1, b_2 , y)  $ 

$ (9)  \;(a_i , b_j \, |\, a_l, b_r , y)  $ for  some $i,l,j,r$ with 
$\{i,l\}=\{1,2\}= \{j,r\}$,  

$(10)  \;(a_i, y \, |\, a_l, b_1, b_2 )  $ for   some $i,l$ with
$\{i,l\}=\{1,2\}$,

$(11) \;(b_i, y \, |\, b_l, a_1, a_2)  $ for   some $i,l$ with
$\{i,l\}=\{1,2\}$,

$(12)  \;(a_i, b_j \, |\, a_l, b_r \, |\, y)  $ for   some $i,l,j,r$ with
$\{i,l\}=\{1,2\}= \{j,r\}$,  

$(13)  \;(a_i, y \, |\, a_l, b_j \, |\, b_r)  $ for   some $i,l,j,r$
with$\{i,l\}=\{1,2\}= \{j,r\}$,  
 
$(14) \; (b_j, y \, |\, b_r, a_i \, |\, a_l)  $ for   some $i,l,j,r$
with$\{i,l\}=\{1,2\}= \{j,r\}$.
 
 \medskip
 
 \underline{Suppose (1) holds.}
By the saturation of $Q$, it  implies at least one of the following: 

$(1.1) \;  (a_1,y \,|\, b_1,b_2\,|\; x) \wedge (a_2,y \,|\, b_1,b_2\,|\; x), $

$(1.2) \; (a_1 ,a_2 \,|\, b_1,y \, | \, x ) \wedge 
(a_1 ,a_2 \,|\, b_2,y \, | \, x ), $

$(1.3) \;  (a_1 ,a_2 \,|\, b_1 , b_2 \,|\, y ) $.

$\bullet $ Suppose  (8) holds.
Since $Q$ is saturated, (8) implies $(a_1,a_2, b_1,b_2, x) \vee 
(a_i ,x \,|\, a_j , b_1, b_2 ) \vee (b_i ,x \,|\, b_j , a_1, a_2 )  $ for some 
$i,j$ with   $\{i,j\}=\{1,2\}$, which contradicts (1).

$\bullet $ Suppose  (9) holds. We can suppose that $i=j=1$ and $l=r=2$ in (9),
so $ (a_1 , b_1 \, |\, a_2, b_2 , y)  $ holds. We get a contradiction by
Lemma \ref{lem: osservazione nuova}, case (C).

$\bullet $ Suppose  (10) holds. We can suppose that $i=1$ in (10), so $ (a_1 , y
\, |\, a_2, b_1 , b_2)  $ holds, but, by the thinness of $Q$, this is impossible.

$\bullet $ Suppose  (11) holds. 
This case is analogous to the previous case (by swapping $(a_1, a_2)$ with 
$(b_1, b_2)$).
 
 $\bullet $ Suppose  (12) holds. We can suppose that $i=j=1$ in (12), so $ (a_1
,b_1 \, |\, a_2, b_2 \,|\, y)  $ holds. 
By the saturation of $Q$,  this implies at least one of the following: 
 
$(12. 1)  \;  (a_1, x  \,|\, a_2,b_2\, |\,y)  \wedge  (b_1, x  \,|\, a_2,b_2\,
|\,y),$ 
 
$ (12.2) \;   (a_1, b_1  \,|\, a_2,x \,|\, y )  \wedge (a_1, b_1  \,|\, b_2,x
\,|\, y ),
   $
  
   $ (12.3)   \;  (a_1, b_1 \,|\, a_2,b_2 \, |\,x)    $.
 
 Observe that (1.1) contradicts  both (12.1) and (12.2), 
 (1.2) contradicts both (12.1) and (12.2),
 (1.3) contradicts (12), and, finally, (12.3) contradicts (1).
 
$\bullet $ Suppose  (13) holds. 
We can suppose $i=r=1$, so $ (y, a_1  \, |\, a_2, b_2 \,|\, b_1)  $ holds. 
   We get a contradiction by Lemma \ref{lem: unione di lemmi}, case $(B)$.

$\bullet $ Suppose  (14) holds. 
This case is analogous to  the previous case (by swapping $(a_1, a_2)$ with 
$(b_1, b_2)$).

 \underline{Suppose (2) holds.}
By the saturation of $Q$, it  implies at least one of: 

$(2.1) \; (a_1,y \,|\, b_1,b_2, x) \wedge (a_2,y \,|\, b_1,b_2, x) , $

$(2.2) \; (a_1 ,a_2 \,|\, \overline{ b_1, b_2 ,  y} )   $.

$\bullet $ Suppose  (8) holds. 
By the saturation of $Q$, it  implies: 

$  (a_i ,x  \,|\, b_1,b_2, a_j) \;\quad \vee \quad\; (b_r,x \,|\, a_1,a_2, b_l) 
\;  \quad \vee \quad 
 \; (a_1 ,a_2 , b_1, b_2 ,  x )   $

for some $i,j$ with $\{i,j\} =\{1,2\}$ and 
some $r,l$ with $\{r,l\} =\{1,2\}$.
All the possibilities contradict (2).

$\bullet $ Suppose  (9) holds. We can suppose $i=j=1$.
So $ (a_1 ,b_1 \, |\, a_2, b_2 , y)  $ holds. 
By the saturation of $Q$, it  implies at least one of the following:

$(9.1) \;  (a_1,x \,|\, a_2,b_2, y) \wedge (b_1,x \,|\, a_2,b_2, y)  $,

$(9.2) \; (a_1 ,b_1 \,|\, \overline{ a_2, b_2 ,  x} )   $.
                   
Observe that (2.2) contradicts (9) and (9.2) contradicts (2). Moreover (2.1)
contradicts (9.1).

$\bullet $ Suppose  (10) holds.  Since $Q$ is thin, we get a contradiction.

$\bullet $ Suppose  (11) holds. 
We can suppose $i=1$, and we get a contradiction by Lemma \ref{lem: osservazione
nuova}, case $(B)$.

$\bullet $ Suppose  (12) holds. We get a contradiction by Lemma \ref{lem:
osservazione nuova}, case (C).

$\bullet $ Suppose  (13) holds. 
We can suppose $i=r=1$, so $ (y, a_1  \, |\, a_2, b_2 \,|\, b_1)  $ holds. 
 We get a contradiction by Lemma \ref{lem: unione di lemmi}, case $(F)$.

$\bullet $ Suppose  (14) holds. We can suppose $j=l=1$, so 
$ (y, b_1  \, |\, a_2, b_2  \, |\,  a_1 ) $ holds. We get a contradiction by Lemma \ref{lem: unione di lemmi}, case $(G)$. 

 \underline{Suppose (3) holds.}
This case is analogous to the case where (2) holds (swap $(a_1 , a_2)$ with
$(b_1 , b_2)$).

 \underline{Suppose (4) holds.} 
 By the saturation of $Q$, it implies at least one of the following: 
 
$(4. 1)  \;  (a_1, y  \,|\, b_1, x \, |\,b_2)  
  \; \wedge   (a_2, y  \,|\, b_1, x \, |\,b_2)   $, 
 
$ (4.2) \;   ( a_1 , a_2 \,|\, b_1, y \,|\, b_2 )   $,
  
$ (4.3)   \;  (a_1 ,a_2 \,|\, b_1, x  \, |\,y)  $.

$\bullet $ Suppose  (8) holds.  By the saturation of $Q$, condition 
(8) implies  $ ( a_1, a_2 , b_1, b_2, x)  \vee  ( a_1, x \, |\,  a_2 , b_1, b_2)
\vee  ( a_2, x \, |\,  a_1 , b_1, b_2) \vee  ( b_1, x \, |\,  a_1 , a_2, b_2) 
\vee  ( b_2, x \, |\,  a_1 , a_2, b_1)  $, which contradicts (4).

$\bullet $ Suppose  (9) holds.  

First case: $ i=j=1$ in (9).  So we have $ ( a_1, b_1  \, |\,  
 a_2, b_2, y) $. By Lemma \ref{lem: unione di lemmi}, case $(F)$ we get a
contradiction.

Second case:  $ i=r=1$ in (9).
So we have $ ( a_1, b_2  \, |\,   a_2, b_1, y) $. We get a contradiction 
by Lemma \ref{lem: unione di lemmi}, case $(G)$.

 $\bullet $ Suppose  (10) holds.  
 We can suppose $i=1$ in (10). Since $Q$ is thin, (4) contradicts (10).

 $\bullet $ Suppose  (11) holds.  
   
 First suppose that $i=1$ in (11). This is impossible by Lemma \ref{lem:
unione di lemmi}, case $(E)$.
 
 Now suppose that $i=2$ in (11). This is impossible by Lemma \ref{lem: unione di
lemmi}, case $(H)$.
 
 $\bullet $ Suppose  (12) holds. We can suppose $i=j=1$. By Lemma
\ref{lem: unione di lemmi}, case $(B)$ we get a contradiction. 
  
 $\bullet $ Suppose  (13) holds. 
 We can suppose $i=1$. 
 If $ r=1$ we get a contradiction by Lemma \ref{lem: unione di lemmi}, case
$(C)$.
 If $r=2$ we get a contradiction by the thinness of $Q$.

 $\bullet $ Suppose  (14) holds. 
 We can suppose $l=1$.
 If $j=1$  we get a contradiction by Lemma \ref{lem: unione di lemmi}, case
$(D)$.
 If $j=2$ we get a contradiction by Lemma \ref{lem: unione di lemmi}, case $(C)$.

 \underline{Suppose (5) holds.}  This case is analogous to the case
 where (4) holds (swap $b_1$ with $b_2$).

 \underline{Suppose (6) holds.}  This case is analogous to the case
 where (4) holds (swap $a_1$ with $b_1$ and  $a_2$ with $b_2$).

 \underline{Suppose (7) holds.}  This case is analogous to the case
 where (6) holds (swap $a_1$ with $a_2$).
 \end{proof}

\bigskip

The next goal is to prove that a quartet system  $S$ as in Definition
\ref{defi:quartet S coming from quintet Q} is in fact
TTS.

\begin{prop}\label{prop: quartets induced by quintets are thin}
 A quartet system  $S$ associated to a TTS quintet system $Q$ 
 as in Definition \ref{defi:quartet S coming from quintet Q} is thin. 
\end{prop}
\begin{proof}
 Assume by contradiction that $(a,b\, |\,c,d)\wedge (a,c\,
|\,b,d)$ holds; by  Proposition \ref{esistequalsiasi}, 
 the hypothesis that $(a,b\, |\,c,d)$ holds 
 is equivalent to 
 say that, for any $y\in X$, we have:
 \begin{align} \label{conditions one}
  (a,b\, |\,\overline{c,d,y}) \lor (c,d\, |\,\overline{a,b,y}) .
 \end{align}
 Moreover, by Definition \ref{defi:quartet S coming from quintet Q},
 the fact that $(a,c\, |\,b,d)$ means that there exists $x\in X$ such that
 \begin{align} \label{conditions two}
  (a,c\, |\,\overline{b,d,x}) \lor (b,d\, |\,\overline{a,c,x}) .
 \end{align}
If we choose $y=x$ in \eqref{conditions one}, we get a contradiction with \eqref{conditions two} by  the thinness of $Q$.
\end{proof}

\begin{prop}\label{prop: quartets induced by quintets are transitive}
 A quartet system  $S$ associated to a TTS quintet system  $Q$ 
 as in Definition \ref{defi:quartet S coming from quintet Q} is transitive. 
\end{prop}
\begin{proof}
 The goal is to prove $(a_1 , a_2 \, |\, b_1 , b_2 ) \wedge (a_2 , a_3 \, |\,
b_1 , b_2 ) \Rightarrow (a_1 , a_3 \, |\, b_1 , b_2 )$. 
 Recall that, by Proposition \ref{esistequalsiasi}, 
 $(a_1 , a_2 \, |\, b_1 , b_2 )$ means that, for every $x\in X$, at least one of the
following conditions must hold:

 $(1.1) \; (a_1 , a_2 \, |\, b_1 , b_2 \, |\,x )$,
 
 $(1.2) \;(a_1  , a_2 \, |\, b_1 , b_2 , x) $,
 
 $(1.3) \; (a_1 , a_2 , x\, |\, b_1 , b_2 )$,
 
 $(1.4) \; (a_1 ,  a_2 \, |\, b_1 , x \, |\, b_2 )$,
 
 $(1.5) \; ( a_1 , a_2 \, |\, b_2 , x \, |\, b_1)$,
 
 $(1.6) \; ( a_1 , x \, |\, b_1 , b_2 \, |\, a_2)$,
 
 $(1.7) \; ( a_2  , x \, |\, b_1 , b_2 \, |\, a_1 ) $.
 
 Similarly, by Proposition \ref{esistequalsiasi}, 
 $(a_2 , a_3 \, |\, b_1 , b_2 )$ means that, for every $x\in X$, at least one of the
following conditions holds:

  $(2.1)\; (a_2 , a_3 \, |\, b_1 , b_2 \, |\,x ) $,
  
  $(2.2)\; (a_2 , a_3 \, |\, b_1 , b_2 , x) $,
  
  $(2.3)\;(a_2 , a_3 , x\, |\, b_1 , b_2 ) $,
  
  $(2.4)\;(a_2 , a_3 \, |\, b_1 , x \, |\, b_2 ) $,
  
  $(2.5)\;( a_2 , a_3 \, |\, b_2 , x \, |\, b_1) $,
  
  $(2.6)\;( a_2 , x \, |\, b_1 , b_2 \, |\, a_3) $,
  
  $(2.7)\;( a_3 , x \, |\, b_1 , b_2 \, |\, a_2 )$. 
 
First we are going to show that, for any $k\in\{1,\ldots , 7\}$, 
if there exists $x $ such that 
$(1.k)\wedge(2.k) $ holds, then $ (a_1 , a_3 \, |\, b_1 , b_2 )$ holds; then we prove that if there exists $x $ such that one of 
the remaining pairings holds, then we get either $(a_1 , a_3 \, |\, b_1 , b_2 )$ or a
contradiction. Observe that, by symmetry, it is sufficient to consider the cases $(1.k)\wedge(2.j) $ with $j >k$.

\underline{Suppose $(1.h)\wedge(2.h)$} for some  $x\in X$ and $h\in\{1,\ldots ,
5\}$; then $(a_1 , a_3 \, |\, b_1 , b_2 ) $ follows from the
transitivity of $Q$ and  Definition \ref{defi:quartet S coming from quintet
Q}. 

\underline{Assume $(1.6)\wedge(2.6) $};
since $Q$ is saturated, $(1.6)$ implies that

  $(x , a_3 \, |\, b_1 , b_2 \, |\, a_2 )\; (1.6.1) \quad\vee\quad (a_1 , x \,
|\, b_1 , a_3 \, |\, a_2 )\;(1.6.2)  
  \quad\vee\quad (a_1 , x \, |\, b_1 , b_2 \, |\, a_3 ) \;(1.6.3) $

and $(2.6)$ implies that 

  $(a_1 , x   \, |\, b_1 , b_2 \, |\, a_3 )\;(2.6.1)\quad\vee\quad (a_2 , x   \,
|\, b_1 , a_1 \, |\, a_3 )\;(2.6.2)
  \quad\vee\quad (a_2 , x   \, |\, b_1 , b_2 \, |\, a_1 ) \;(2.6.3)$. 

Each of $(1.6.1)\wedge(2.6)$, $(2.6.3)\wedge(1.6)$,
$(1.6.2)\wedge(2.6.2)$ contradicts the thinness of $Q$; moreover, by Definition \ref{defi:quartet S coming from quintet Q},  each of $(1.6.3)$ and  $(2.6.1)$  implies
$(a_1 ,a_3 \, |\, b_1 , b_2 ) $ and this allows us to conclude.

\underline{The case $(1.7)\wedge(2.7)$} can be recovered from the previous one
by swapping $a_1$ with $a_3$. 

\underline{The case $(1.1)\wedge(2.2) $} is 
impossible by Lemma \ref{lem: unione di lemmi}, case $(E)$.

\underline{Assume $(1.1)\wedge(2.3)$};
since $Q$ is saturated, from $(2.3)$ we get at least one of the following: 

  $(2.3.1)\;(b_1 , a_1 \, |\, a_2 , a_3  ,  x   ) $,
  
  $(2.3.2)\;(b_1 , b_2 \, |\, x   , a_3  ,  a_1 ) $,
  
  $(2.3.3)\;(b_1 , b_2 \, |\, x   , a_3 \, |\, a_1 ) $,
  
  $(2.3.4)\;(b_1 , b_2 \, |\, x   , a_1 \, |\, a_3 ) $,  

  $(2.3.5)\;(b_1 , b_2 \, |\, a_1 , a_3 \, |\, x   ) $.  
  
The occurrence of $(1.1) \wedge (2.3.1)$ is impossible
by Lemma \ref{lem: unione di lemmi}, case $(F)$.

Each of $(2.3.2)$, $(2.3.3)$, $(2.3.4)$, $(2.3.5)$ allows to
conclude, by 
 Definition \ref{defi:quartet S coming from quintet Q}, that $(a_1 , a_3 \,|\, b_1 , b_2 )$ holds.

\underline{The case $(1.1)\wedge(2.4) $} is
impossible by Lemma \ref{lem: unione di lemmi}, case $(D)$; swapping 
$b_1$ with $b_2$, we also get that \underline{the case $(1.1)\wedge(2.5) $} is
impossible.

\underline{Suppose $(1.1)\wedge(2.6) $};
since $Q$ is saturated, $(1.1)$ implies 

  $(a_1 , a_3 \, |\, b_1 , b_2 \, |\, x   )\;(1.1.1)\quad\vee\quad
  (a_1 , a_2 \, |\, b_1 , a_3 \, |\, x   )\;(1.1.2)\quad\vee\quad
  (a_1 , a_2 \, |\, b_1 , b_2 \, |\, a_3 ) \;(1.1.3)$

and  $(2.6)$ implies one of    $(2.6.1)$, $(2.6.2)$, $(2.6.3)$  above.
From $(1.1.1)$ as well as from $(1.1.3)$ and  from $(2.6.1)$ 
one can deduce,  by means of Definition
\ref{defi:quartet S
coming from quintet Q}, that $(a_1 , a_3 \, |\, b_1 , b_2 ) $ holds;
each  of $(2.6.3)\wedge (1.1)$ and $(1.1.2)\wedge (2.6.2)$
contradicts the thinness of $Q$ and this allows us to conclude.

\underline{Suppose $(1.1)\wedge(2.7) $};
since $Q$ is saturated, $(1.1)$ implies one of    $(1.1.1)$, $(1.1.2)$, $(1.1.3)$ and $(2.7)$ implies 

  $(a_1 , a_3 \, |\, b_1 , b_2 \, |\, a_2   )\;(2.7.1)\quad\vee\quad
  (a_3 , x \, |\, b_1 , a_1 \, |\, a_2   )\;(2.7.2)\quad\vee\quad
  (a_3 , x \, |\, b_1 , b_2 \, |\, a_1 ) \;(2.7.3)$.

From $(1.1.1)$ as well as from $(1.1.3)$  and
from  $(2.7.3)$ one can deduce that 
$(a_1 , a_3 \, |\, b_1 , b_2 ) $ holds. Finally, 
the case $(2.7.1)\wedge(1.1)$, as well as the case $(1.1.2)\wedge(2.7.2)$, contradicts the thinness of $Q$ and so we can conclude.

\underline{The case $(1.2)\wedge(2.3) $} is impossible by the thinness of $ Q$.

\underline{The case $(1.2)\wedge(2.4) $} is
impossible by  Lemma \ref{lem: unione di lemmi}, case $(E)$; swapping
$b_1$ with $b_2$, we also have that \underline{the case $(1.2)\wedge(2.5) $}
cannot hold.
 
\underline{The case $(1.2)\wedge(2.6)$} is impossible by the thinness of $Q$.

\underline{The case $(1.2)\wedge(2.7) $} cannot hold by Lemma \ref{lem:
unione di lemmi}, case $(H)$.

\underline{The case $(1.3)\wedge (2.4) $} cannot hold by Lemma
\ref{lem: unione di lemmi}, case $(G)$.

\underline{The case $(1.3)\wedge(2.5) $}
is analogous to the previous one swapping $b_1$ with $b_2$.

\underline{Assume $(1.3)\wedge(2.6) $};
since $Q$ is saturated,  $(1.3)$ implies 

  $(b_1 , a_3 \, |\, a_1 , a_2 , x   )\;(1.3.1)\quad\vee\quad
  (b_1 , b_2 \, |\, \overline{  a_1 , a_2, a_3 }  )\;(1.3.2)$

and,
similarly, $(2.6)$ implies one of $(2.6.k)$, $k=1,2,3$.
From $(2.6.1)$, as well as from $(1.3.2)$,
one can deduce  that $(a_1 , a_3 \, |\, b_1  , b_2 ) $ holds.
Moreover, $(2.6.3)\wedge(1.3)$, as well as  $(2.6.2)\wedge(1.3.1)$, contradicts the thinness of $Q$ and so we conclude.


\underline{Assume $(1.3)\wedge(2.7) $}. Since $Q$ is saturated,  
 $(1.3)$ implies one of $(1.3.1)$, $(1.3.2)$ and $(2.7) $ implies 
one of $(2.7.1)$, $(2.7.2)$, $(2.7.3)$. Observe that each of 
$(1.3.2)$, $(2.7.1)$, $(2.7.3)$ implies $(a_1 , a_3 \, |\, b_1  , b_2 ) $. Moreover,   $(1.3.1) \wedge (2.7.2)$
contradicts the thinness of $Q$.

\underline{Cases $(1.4)\wedge (2.5) $, $(1.4)\wedge (2.6) $, $(1.4)\wedge (2.7) $}  are
impossible by  Lemma \ref{lem: unione di lemmi}, respectively case $(D)$, case $(B)$, case $(C)$.

\underline{Assume $(1.5)\wedge (2.6)$};
then by swapping $b_1$ with $b_2$ one gets back to the case
$(1.4)\wedge(2.6)$.

\underline{Suppose $(1.5)\wedge (2.7) $}; 
then by swapping $b_1$ with $b_2$ one gets back to the case
$(1.4)\wedge(2.7)$.

\underline{Assume $(1.6)\wedge (2.7) $};
then, by transitivity of $Q$, one gets  $(a_1 , a_3 \, |\, b_1 , b_2 \, |\,
a_2 )$, and by Definition \ref{defi:quartet S coming from quintet Q}
one can deduce $(a_1 , a_3 \, |\, b_1 , b_2 ) $.
\end{proof}

\begin{prop}\label{prop: quartets induced by quintets are saturated}
 A quartet system  $S$ associated to a TTS quintet system $Q$
 as in  Definition \ref{defi:quartet S coming from quintet Q} is saturated. 
\end{prop}
\begin{proof}
 Suppose that $(a_1 , a_2 \, |\, b_1  , b_2) $ and fix $x\in X$. 
 By Definition \ref{defi:quartet S coming from quintet Q}
 there exists  $z \in X$ such that at least one of the following holds:
 
  $(1)\; (z\, |\,a_1 , a_2 \, |\, b_1 , b_2 )$,
  
  $(2)\;(a_1  , a_2 \, |\, b_1 ,  b_2 , z)$,
  
  $(3)\;(a_1 , a_2 \, |\, b_1 , z \, |\, b_2 )$,
  
  $(4)\;(a_1 , a_2 \, |\, b_2 , z \, |\, b_1 )$,
  
  $(5)\;( b_1 , b_2 \, |\, a_1 , z \, |\, a_2)$,
  
  $(6)\;( b_1 , b_2 \, |\, a_2 , z \, |\, a_1)$,
  
  $(7)\;( b_1 , b_2 \, |\,a_1 , a_2  , z )$. 
 
The argument consists in showing that any of the  items above implies either $(a_1 , a_2
\, |\, b_1 , x)$ or  $(a_1 , x\, |\, b_1 , b_2)$; since it is repetitive, and uses essentially only Definition
\ref{defi:quartet S coming from quintet Q}, we only give a sample of the whole
argument. 

Suppose that $(4)$ holds; then, as $Q$ is saturated, we have:

 $ (a_1 , a_2 \, |\, b_2 , z \, |\, x)\;(4.1)\quad \vee \quad (a_1 , x \, |\,
b_2  , z \, |\, b_1)\;(4.2) \quad \vee \quad (a_1 , a_2 \, |\, b_2 , x \, |\,
b_1)\;(4.3) $. 

By $(iv)$ of Definition \ref{defi:quartet S coming
from quintet Q}
with $a=a_1, \, b=a_2 ,\, c=x,\, d=b_2 ,\, y=z $, condition $(4.1)$  entails  $(a_1 , a_2 \,
|\, x , b_2) $; since
also $(a_1 , a_2 \, |\, b_1 , b_2) \in S$, the hypothesis that $S$ is transitive
allows us to get $(a_1 , a_2 \, |\, b_1 , x)$.

By $(iv)$ of Definition \ref{defi:quartet S coming from quintet Q}
with $a=a_1, \, b=x ,\, c=b_1,\, d=b_2 , \, y=z $, case $(4.2)$  entails $(a_1 , x \,|\, b_1 ,b_2)$.

By $(iv)$ of Definition \ref{defi:quartet S coming
from quintet Q}
with $a=a_1, \, b=a_2 ,\, c=b_1,\, d=x ,\, y=b_2 $, case $(4.3)$  entails $(a_1 , a_2 \, |\, b_1 , x)$ as wanted.
\end{proof}

\begin{rem} \label{Q'S} 
Let $S$ be the quartet system of a phylogenetic $X$-tree $(T, \varphi)$ 
 and let $Q'$ be the quintet system in $X$ associated to $S$. Then $Q'$ is the
quintet system of $(T, \varphi)$. 
\end{rem}

\begin{rem}  \label{abcdinS} Let 
$Q$ be a TTS quintet system in $X$;
call $S$ the quartet system associated to $Q$. We have that  $(a,b,c,d) \in S$
if and only if
\begin{equation} \label{ma} (a,b,c,d,x) \in Q \vee (a,x \,| \, b, c,d) \in Q 
\vee (b,x \,| \, a, c,d) \in Q 
 \vee (c,x \,| \, a,b,d) \in Q  \vee (d,x \,| \, a, b, c) \in Q  \end{equation}
  for any $x \in X$. 
 By Proposition \ref{esistequalsiasi} this holds if and only if 
there exists  $x \in X$ such  that (\ref{ma}) holds.
\end{rem}

\begin{prop} \label{Q'=Q}
Let $Q$ be a TTS quintet system in $X$;
call $S$ the quartet system associated to $Q$ and
$Q'$ the quintet system associated to $S$. Then $Q=Q'$.
\end{prop}

\begin{proof}
$\bullet $ First we prove that every partition of  $Q'$ of kind $(2,2,1)$    or
of kind $(2,3)$  is also an element of  $ Q$.

Let $(a,b \, |\, c,d \,|\, e) \in Q'$. Suppose that  $(a,b \, |\, c,d \,|\, e)
\not\in Q$.
Thus one of the following conditions holds:

1) $(a,b,c,d,e) \in Q$; by the definition of $S$, this would imply $(a,b,c,d) 
\in S $, which is absurd since, by the definition of $Q'$, 
we have that $(a,b \, |\, c,d ) \in S$ (and $S$ is thin).

2) $(x,y\,| \, z, w,u) \in Q$ for $\{x,y,z,w,u\} =\{a,b,c,d,e\}$; by the
definition of $S$, this would imply $(x,z ,w,u) \in S$, which is absurd since,
by the definition of $Q'$, 
we have that a partition of kind $(2,2) $ of   $\{x,z,w,u\}$  is in $S$ (and $S$
is thin).

3) A partition of kind $(2,2,1) $
 of   $\{a,b,c,d,e\}$, different from $(a,b \, |\, c,d \,|\, e) $,  is in $Q$;
up to swapping $a$ with $b$ or $c$ with $d$ or $\{a,b\}$ with $\{c,d\}$, we can
suppose  $(a,c\,| \, b, d \,| \, e) \in Q$ or $(a,e\,| \, c, d \,| \, b) \in Q$
or $(a,e\,| \, b, d \,| \, c) \in Q$.
    
By the definition of $S$, $(a,c\,| \, b, d \,| \, e) \in Q$ would imply $(a,c\,
|\, b,d) \in S$, which is absurd since, by the definition of $Q'$, 
we have that $(a,b \, |\, c,d ) \in S$ (and $S$ is thin).

By the definition of $S$, $(a,e\,| \, c, d \,| \, b) \in Q$ would imply $(a,e\,
|\, b,c) \in S$, which is absurd since, by the definition of $Q'$, 
we have that $(a,b \, |\, c,e ) \in S$ (and $S$ is thin).

By the definition of $S$, $(a,e\,| \, b, d \,| \, c) \in Q$ would imply $(a,c\,
|\, b,d) \in S$, which is absurd since, by the definition of $Q'$, 
we have that $(a,b \, |\, c,d ) \in S$ (and $S$ is thin).

Let $(a,b \, |\, c,d , e) \in Q'$. Suppose that  $(a,b \, |\, c,d , e) \not\in
Q$.
Thus one of the following conditions holds:

1) $(a,b,c,d,e) \in Q$; by Remark \ref{abcdinS}, this would imply $(a,b,c,d) 
\in S $, which is absurd since, by the definition of $Q'$, 
we have that $(a,b \, |\, c,d ) \in S$ (and $S$ is thin).

2) A partition of kind $(2,3) $
 of   $\{a,b,c,d,e\}$, different from $(a,b \, |\, c,d , e) $,  is in $Q$; up to
making a permutation of   $\{a,b\}$ or of $\{c,d,e\}$, we can suppose  $(a,c\,|
\, b, d , e) 
 \in Q$ or $(c,d\,| \, a , b, e) \in Q$.
  
  The condition $(a,c\,| \, b, d , e) \in Q$ would imply $(a,c \, |\, b,d) 
   \in S$, which is absurd since, by the definition of $Q'$, we have that 
    $(a,b \, |\, c,d) \in S$ 
 
  The condition $(c,d\,| \, a , b, e)  \in Q$ would imply $(c,d \, |\, b,e) 
   \in S$, which is absurd since, by the definition of $Q'$, we have that 
    $(b ,c,d,e) \in S$.

3) A partition of kind $(2,2,1) $
 of   $\{a,b,c,d,e\}$  is in $Q$; 
 up to making a permutation of   $\{a,b\}$ or of $\{c,d,e\}$, we can suppose
 $(a,b\,| \, c, d \,| \, e) \in Q$ or $(a,c\,| \,  d ,e \,| \, b) \in Q$ or
$(a,c\,| \, b, d \,| \, e) \in Q$ or $(c,d\,| \, b, e \,| \, a) \in Q$.
 
  The condition   $(a,b\,| \, c, d \,| \, e) \in Q$   would imply $(b,e \, |\,
c,d) 
   \in S$, which is absurd since, by the definition of $Q'$, we have that 
    $(b , c,d,e) \in S$. 
    
  The condition   $(a,c\,| \,  d ,e \,| \, b) \in Q$   would imply $(a,c \, |\,
b,d) 
   \in S$, which is absurd since, by the definition of $Q'$, we have that 
    $(a,b \,|\, c,d) \in S$. 
    
  The condition   $(a,c\,| \, b, d  \,| \, e) \in Q$   would imply $(a,c \, |\,
b,d) 
   \in S$, which is absurd since, by the definition of $Q'$, we have that 
    $(a,b \,|\, c,d) \in S$. 
    
  The condition   $(c,d\,| \, b, e \,| \, a) \in Q$  would imply $(c,d \, |\,
b,e) 
   \in S$, which is absurd since, by the definition of $Q'$, we have that 
    $(c,d,b,e) \in S$.

$\bullet $ Let us  prove  that every partition of  $Q$ of kind $(2,2,1)$ or of
kind $(2,3)$     is also an element of  $ Q'$.

Let $(a,b \, |\, c,d \,|\, e) \in Q$. 
By the definition of $Q'$, we have  that   $(a,b \, |\, c,d \,|\, e) \in Q'$ if
and only if  $(a,b \, |\, c,d ) \in S \wedge (a,b \, |\, c,e) \in S \wedge (a,e
\, |\, c,d ) \in S  $ 
and this follows from the fact that $(a,b \, |\, c,d \,|\, e) \in Q$ and the
definition
of $S$.

Let $(a,b \, |\, c,d , e) \in Q$. 
By the definition of $Q'$, we have  that   $(a,b \, |\, c,d , e) \in Q'$ if and
only if  $(a,b \, |\, c,d ) \in S \wedge (a,b \, |\, c,e) \in S \wedge (a,e ,
c,d ) \in S 
\wedge (b , c,d ,e) \in S  $ 
and this follows from the fact that $(a,b \, |\, c,d , e) \in Q$ and the
definition
of $S$.
\end{proof}

\begin{thm} Let $Q$ be a quintet system in $X$; we have that  $Q$ is the quintet
system of a phylogenetic $X$-tree
if and only if $Q$ is TTS.
\end{thm}

\begin{proof}
$\Rightarrow$ Very easy to prove.

$\Leftarrow$  
By Propositions \ref{prop: quartets induced by quintets are thin}, \ref{prop:
quartets induced by quintets are saturated} and \ref{prop: quartets induced by
quintets are transitive}, the quartet system $S$ associated to $Q$ is TTS. So,
by Theorem \ref{CS},
there exists an $X$-tree $(T, \varphi)$  whose quartet system is $S$. Let $Q'$
be the quintet system associated to $S$.  By Remark \ref{Q'S}, we have 
that $Q'$ is the quintet system associated to $(T, \varphi)$. By Proposition
\ref{Q'=Q},   we have that $Q=Q'$, hence $Q$  
 is the quintet system associated to $(T, \varphi)$.
\end{proof}

{\bf Address of both authors:}
Dipartimento di Matematica e Informatica ``U. Dini'',

viale Morgagni 67/A,
50134  Firenze, Italia

{\bf
E-mail addresses:}
scalamai@math.unifi.it, rubei@math.unifi.it

\end{document}